\documentclass{article}
\usepackage[utf8]{inputenc}
\usepackage{setspace}

\usepackage{geometry}
\usepackage{graphicx}	
\usepackage{subcaption}
\usepackage{float} 
\usepackage[toc]{appendix}

\usepackage[draft]{todonotes}   

\usepackage{enumerate}	
\usepackage{paralist}		

\usepackage{amssymb}
\usepackage{amsmath, amsthm, amssymb}

\newtheorem{theorem}{Theorem}[section]

\newtheorem{proposition}[theorem]{Proposition}
\newtheorem{claim}[theorem]{Claim}
\newtheorem{lemma}[theorem]{Lemma}
\newtheorem{corollary}[theorem]{Corollary}

\newtheorem{remark}{Remark}

\theoremstyle{definition}

\newtheorem{definition}[theorem]{Definition}

\def\HH{\mathcal{H}}
\def\G{\mathcal{G}}

\def\C{\mathcal{C}}
\def\Z{\mathbb{Z}}
\def\Zp{\mathbb{Z}/p\mathbb{Z}}

\usepackage{tikz}
\newcounter{casenum}

\newcommand*{\rom}[1]{\expandafter{\romannumeral #1\relax}}

\usepackage{authblk}

\usepackage{tikz}
\tikzstyle{every node}=[circle, draw, fill=black!50, inner sep=0pt, minimum width=4pt]

\title{On the number of sum-free triplets of sets}

\author{ Igor Araujo 
	\footnote{Department of Mathematics, University of Illinois at Urbana-Champaign, Urbana, Illinois 61801, USA. Email: \texttt{igoraa2@illinois.edu}. Research supported by Arnold O. Beckman Research Award (UIUC Campus Research Board RB 18132).}
	\qquad J\'ozsef Balogh
	\footnote{Department of Mathematics, University of Illinois at Urbana-Champaign, Urbana, Illinois 61801, USA, and Moscow Institute of Physics and Technology, Russian Federation. E-mail: \texttt{jobal@illinois.edu}. Research supported by NSF RTG Grant DMS-1937241, NSF Grant DMS-1764123, Arnold O. Beckman Research Award (UIUC Campus Research Board RB 18132), the Langan Scholar Fund (UIUC), and the Simons Fellowship.}
	\qquad Ramon I. Garcia
	\footnote{Department of Mathematics, University of Illinois at Urbana-Champaign, Urbana, Illinois 61801, USA. Email: \texttt{rig2@illinois.edu}.}
}

\begin{document}
	
	\maketitle
	
	\begin{abstract}
		We count the ordered sum-free triplets of subsets in the group $\Zp$, i.e., the triplets $(A,B,C)$ of sets $A,B,C \subset \Zp$ for which the equation $a+b=c$ has no solution with $a\in A$, $b \in B$ and $c \in C$. Our main theorem improves on a recent result by Semchankau, Shabanov, and Shkredov \cite{Semchankau} using a different and simpler method. Our proof relates previous results on the number of independent sets of regular graphs by Kahn \cite{Kahn}, Perarnau and Perkins~\cite{Perarnau-Perkins}, and Csikvári \cite{Csikvari} to produce explicit estimates on smaller order terms. We also obtain estimates for the number of sum-free triplets of subsets in a general abelian group. 
	\end{abstract}
	
	\section{Introduction}
	
	Let $p$ be a prime number and $\Zp$ the abelian (additive) group of integers modulo $p$. We estimate the number of ordered triplets $A,B,C\subseteq \Zp$ such that there is no triplet $a\in A$, $b\in B$ and $c\in C$ with $a+b=c$. We call such $(A,B,C)$ a \emph{sum-free triplet}. 
	Consider the auxiliary 3-uniform 3-partite hypergraph $\HH$ with vertex set $X \cup Y \cup Z$, each of $X,Y,Z$ being disjoint and a copy of $\Zp$. We have an edge $\{x,y,z\}$ for $x\in X$, $y\in Y$, $z \in Z$ when $x+y=z$. The number of ordered triplets $A,B,C\subseteq \Zp$ as above is the number of independent sets $A \cup B \cup C$ in $\HH$, where $A \subset X$, $B \subset Y$, and $C \subset Z$. For simplicity, we denote by $i(H)$ the number of independent sets in the (hyper)graph~$H$.   
	
	As suggested by Semchankau, Shabanov, and Shkredov in Remark 18 of \cite{Semchankau}, we try to obtain a bound of the form 
	\begin{equation} \label{eq:main}
		i(\HH) = 3\cdot 4^p + 3p \cdot 3^p + Q_2(p) \lambda_2^p + Q_3(p) \lambda_3^p + \ldots + O((\lambda_\ell-c_*)^p),
	\end{equation}
	where $Q_2(p), \ldots, Q_{\ell-1}(p)$ are some polynomials in $p$. The motivation for \eqref{eq:main} comes from considering the separate cases when $\min \{ |A|, |B|, |C| \} = k$. For this, we fix $A$ with $|A|=k$ and count the ways of completing the sum-free triplet $(A,B,C)$. However, we do not give precise estimates for all possible choices of $A$ when $k\ge 3$. Therefore, we only obtain the upper and lower bounds below.
	
	\begin{theorem} \label{thm:main}
		The number of sum-free triplets $A,B,C\subseteq \Zp$ is at most
		$$ 3\cdot 4^p + 3p \cdot 3^p + 3\binom{p}{2} \left( \frac{1+\sqrt{5}}{2} \right)^{2p} + Q_{3,1}(p) \cdot (1+\sqrt{2})^p + Q_{3,2}(p) \cdot 458^{p/7} + 3\binom{p}{4} (1+o(1)) \cdot 31^{p/4}, $$
		and at least
		$$ 3\cdot 4^p + 3p \cdot 3^p + 3\binom{p}{2} \left( \frac{1+\sqrt{5}}{2} \right)^{2p} + Q_{3,1}(p) \cdot (1+\sqrt{2})^p + Q_{3,2}(p) \cdot 2.38898^{p},$$
		where $Q_{3,1}(p)=  3 \cdot \frac{3}{p-2} \cdot \binom{p}{3} = 3 \cdot \binom{p}{2}$, and $Q_{3,2}(p) = 3 \cdot \frac{p-5}{p-2} \cdot \binom{p}{3}$.
	\end{theorem}
	
	\begin{remark} \normalfont
		The bases of the exponential terms on Theorem \ref{thm:main} are 
		$4>3> \left( \frac{1+\sqrt{5}}{2} \right)^2 \approx 2.6180 > 1+\sqrt{2} \approx 2.4142 > 458^{1/7} \approx 2.3995 > 2.38898 > 31^{1/4} \approx 2.3596$.
	\end{remark}

	\begin{remark} \normalfont
		The constant 2.38998 in Theorem \ref{thm:main} is a lower bound on the value $\alpha + \alpha^{-1} -2$, where $\alpha$ is the unique real root of $\alpha(1-\alpha)^2 = (1-2\alpha)^3$. See Theorem \ref{thm:csikvari} for more details.
	\end{remark}
	
	\begin{remark} \normalfont
		One could believe that, for each $i$, there is one $\lambda_i$ as in \eqref{eq:main} corresponding to the case $\min \{ |A|, |B|, |C| \} = i$. The proof of Theorem \ref{thm:main} below contradicts this idea, by showing that when $\min \{ |A|, |B|, |C| \}=3$ we have at least two distinct exponential terms, all of them greater than the possible exponential terms for $\min \{ |A|, |B|, |C| \} \ge 4$. 
	\end{remark}
	
	\begin{remark} \normalfont
		The proof of Theorem \ref{thm:main} uses the structure of $\Zp$. In Section \ref{sec:generalgroup}, we obtain similar estimates for a general abelian group $G$, see Theorems \ref{thm:generalcase} and \ref{thm:p1}.
	\end{remark}
	
	Let $(G,+)$ be an additive abelian group of order $N$. Using the same methods as in the case $G=\Zp$, we get the following result.
	
	\begin{theorem}\label{thm:generalcase}
		Fix $\varepsilon \in (0,1)$. Then for sufficiently large $N$ the following holds. Let $(G,+)$ be an additive abelian group of order $N$. The number of ordered triplets $A,B,C\subseteq G$ such that there is no triplet $a\in A$, $b\in B$ and $c\in C$ with $a+b=c$ is at most
		$$ 3\cdot 4^N + 3N \cdot 3^N + 2^{(3/2+3\varepsilon)N} . $$
	\end{theorem}
	
	\begin{remark} \normalfont
		In contrast with Theorem \ref{thm:main}, in the proof of Theorem \ref{thm:generalcase} we will see that the we have more triplets corresponding to the case $k \ge \varepsilon N$ than from the case $k=2$. Indeed, the number of triplets from the case $k=2$ is bounded by $3\binom{N}{2} 7^{N/2} \ll 2^{3N/2}$, since $7^{1/2} \approx 2.64575 < 2^{3/2} \approx 2.8284$ .
	\end{remark}
	
	\begin{remark} \normalfont
		Theorem \ref{thm:generalcase} is sharp for $G=(\Z / 2\Z)^n$, when we have at least 
		$$3\cdot 4^N + 3N \cdot 3^N + 3\binom{N}{2} 7^{N/2}+(\log_2 N) \cdot 2^{3N/2} - O((\log_2 N)^2 \cdot 2^{3N/4}) $$
		triplets, where $N=2^n$. The first three terms come from the cases of $\min(|A|,|B|,|C|)=0,1,2$, respectively. The last terms come from, as in Remark \ref{rem:index2}, the $n$ different groups of index $2$ in $(\Z / 2\Z)^n$ and taking overcounting into consideration.
	\end{remark}
	
	\begin{remark} \normalfont
		In \cite{Semchankau}, Semchankau, Shabanov, and Shkredov proved that the number of sum-free triplets is 
		$$3\cdot 4^N + 3N \cdot 3^N + O((3-c_*)^N),$$
		for some absolute constant $c_*>0$. Theorems \ref{thm:main} and \ref{thm:generalcase} is an improvement on that estimates, as we can explicit compute the constant $c_*$ and the base of the next exponential term.
	\end{remark}
	
	\section{Overview of the paper}
	
	We will prove Theorem \ref{thm:main} by considering the separate cases when $\min \{ |A|, |B|, |C| \} = k$. For this, for various values of $k$, we fix $A$ with $|A|=k$ and count the ways of completing the sum-free triplet $(A,B,C)$. The link graph of $A$, denoted by $\HH_A$, is the bipartite graph with vertex set $V_1\cup V_2$ and edge set the pairs $\{y,z\}$ where $x+y=z$ for some $x \in A$. For each fixed small $A$, we will count the independent sets $i(\HH_A)$ in the link graph of $A$. 
	After taking overcounting into consideration, we conclude that the number of sum-free triplets $(A,B,C)$ with $\min \{ |A|, |B|, |C| \} = k$ is
	$Q_k(p) i(\HH_A) + O(p^{2k} \cdot 2^p)$ for some polynomial $Q_k$ in $p$.
	
	The first cases, when $k=0$ or $1$, were dealt with in \cite{Semchankau}, and we include them here for completeness. 
	In Section~\ref{sec:smallercases}, we precisely (up to the first-order term) count the independent sets for each of the cases $k \le 3$. From the cases $k\le 2$, the number of sum-free triplets is 
	\begin{equation}\label{eq:caseZp,k=0,1,2}
		3\cdot 4^p + 3p \cdot 3^p + 3\binom{p}{2} \cdot \left( \frac{1+\sqrt{5}}{2} \right)^{2p} + O(p^42^p).
	\end{equation}
	
	For $k=3$, we will see that the main term is $Q_{3,1}(p) \cdot (1+\sqrt{2})^p$ for some polynomial $Q_{3,1}(p)$, and it comes from choices of $A$ for which the link graph $\HH_A$ contains cycles $C_4$ as a subgraph. To show that, we apply a general result (namely, Theorem \ref{thm:perkins}) for 3-regular graphs with girth at least 5 by Perarnau and Perkins \cite{Perarnau-Perkins} to upper bound the number of independent sets in the case the link graph does not contain $C_4$ as a subgraph. We obtain a lower bound on the number of independent sets by using a general result (namely, Theorem \ref{thm:csikvari}) for regular graphs by Csikvári~\cite{Csikvari}.
	
	In Section \ref{sec:smallcase}, we deal with the cases $4 \le k \le \varepsilon p$ for any fixed $\varepsilon < 1/10^3$. We make use of an upper bound for general regular graphs (see Theorem \ref{thm:kahn}), by Kahn \cite{Kahn}, to conclude that the number of sum-free triplets, in this case, is at most $3\binom{p}{4} (1+o(1)) \cdot 31^{p/4}$, lower than the number of sum-free triplets for $k = 3$.
	
	In Section \ref{sec:bigcase}, we use the method of hypergraph containers to give an upper bound on the number of independent sets in the remaining case $k \ge \varepsilon p$. The idea is to use a supersaturation result for the number of sums $x+y=z$ with $x\in A$, $y \in B$, and $z\in C$ when $A,B,C$ are large, together with the Hypergraph Container Lemma (Theorem \ref{thm:container}) to obtain an upper bound on the number of independent sets of $\HH$ with a large intersection with each of the parts $V_0$, $V_1$ and $V_2$.
	Our method generalizes for other groups as well. Thus, we show the estimates for general abelian groups $G$ in Section \ref{sec:bigcase}, and we briefly discuss the case when $G=\Zp$ in the beginning of Section \ref{sec:proof}, where we conclude the proof of Theorem \ref{thm:main}. 
	
	Finally, in Section \ref{sec:generalgroup}, we prove Theorem \ref{thm:generalcase} and discuss the differences for the estimates for general abelian groups.
	We highlight that if one knows the number of elements of each order in $G$, a careful look at our method can lead us to precise estimates in the case $k=2$. Subsequently, we examine the behavior of the estimates we obtain for different groups $G$, depending on the prime factorization of its order.
	
	\section{Smaller cases: $k \le 3$} \label{sec:smallercases}
	
	When $|A|=0$, every choice of $B$ and $C$ works. In this case, $i(\HH_A)=4^p$. In total, we have $3 \cdot 4^p - 3 \cdot 2^p + 1$ triplets with at least one of the sets empty. 
	
	When $|A|=1$, the link graph $\HH_A$ is a matching of size $p$, which has $3^{p}$ independent sets. In total, we have $3p\cdot (3^p-1) - 3p^2 \cdot (2^{p-1}-1) + p^2(p-1)$ triplets with $\min \{ |A|, |B|, |C| \} = 1$.
	
	Note that, for each fixed $k$, we need to subtract a factor of $O(p^{2k} \cdot 2^p)$ from the counting of triplets $(A,B,C)$ to avoid overcounting the cases when at least two of the sets $A$, $B$ or $C$ have at most $k$ elements. For the sake of simplicity, from now on, we will focus only on counting $i(\HH_A)$. Hence, we will obtain a bound of the form $Q_k(p) i(\HH_A) + O(p^{2k} \cdot 2^p)$ for some polynomial $Q_k$ in $p$. In what follows, $k=|A|$, and we count the independent sets $i(\HH_A)$ in the link graph of $A$. 
	
	When $|A|=2$, the link graph $\HH_A$ is a Hamiltonian cycle. We need the following easy claim.
	
	\begin{claim} \label{claim:indset_cycle}
		The number of independent sets in a cycle of length $2n$ is $\phi^{2n} + \phi^{-2n}$, where $\phi = \frac{1+\sqrt{5}}{2}$.
	\end{claim}
	
	\begin{proof}
		Denote  by $F_{m+2}$ the number of independents sets in an $m$-vertex path. We break the counting into cases according to whether a fixed vertex is in the independent set or not. The number of independent sets in the cycle is then $F_{2n+1}+F_{2n-1}$. To count the independent sets in a path we repeat the procedure, breaking into cases of when an endpoint of the path is in the independent set or not. We obtain $F_{m+2}=F_{m+1}+F_m$, $F_2=1$, and $F_3=2$. Therefore, $F_m$ is the $m$-th Fibonacci number. 
		
		In this case, $i(\HH_A) = \frac{\phi^{2n+1}+\phi^{-2n-1}}{\sqrt{5}} + \frac{\phi^{2n-1}+\phi^{-2n+1}}{\sqrt{5}} = \phi^{2n} + \phi^{-2n}$, where we used $\phi+\phi^{-1}=\sqrt{5}$ for the last equality. 
	\end{proof}
	
	\noindent In total, we have 
	$ 3\binom{p}{2} \cdot \left( \frac{1+\sqrt{5}}{2} \right)^{2p} + O(p^{4} 2^{p})$
	triplets with $\min \{ |A|, |B|, |C| \} = 2$. This implies that $Q_2(p) = \frac{3p(p-1)}{2} $ and $\lambda_2 = \left( \frac{1+\sqrt{5}}{2} \right)^2 \approx 2.6180$ on \eqref{eq:main} above.
	
	\subsection{Case $k=3$}
	
	When $|A|=3$, we assume without loss of generality that $A=\{0,1,x\}$ and break into cases depending on whether $\HH_A$ contains a $C_4$ or not. For this, let us first prove the following.
	
	\begin{claim} \label{claim:c4}
		Let $A=\{0,1,x\}$. The link graph $\HH_A$ contains $C_4$ as a subgraph if and only if $x\in \{-1,2,\frac{p+1}{2}\}$.
	\end{claim}
	
	\begin{proof}
		We write $(Y,t)$ for the vertex corresponding to $t \in \Zp$ in $B$ and, similarly, $(Z,t)$ for $t \in \Zp$ in $C$. Since the graph $\HH_A$ is vertex transitive, it is enough to check when $(Y,1)$ is in a $C_4$. The neighborhood of $(Y,1)$ is $(Z,1), (Z,2), (Z,x+1)$.
		The second neighborhood of $(Y,1)$ is $(Y,0), (Y,1-x), (Y,2), (Y,2-x), (Y,x), (Y,x+1)$. A vertex is contained in a $C_4$ if and only if two distinct neighbors have one extra common neighbor. Therefore, $(Y,1)$ is contained in a $C_4$ if and only if two of the above second neighbors coincide. This is equivalent to one of the following equations to hold (we include trivial equations here only for the sake of completeness of the proof)
		\begin{align*}
			0=& \, 2 & \, 1-x=& \, 2 &\, 2= & \, x \\
			0=& \, 2-x & \, 1-x=& \, 2-x &\, 2= & \, x+1 \\
			0=& \, x & \, 1-x =& \, x &\, 2-x = & \, x \\
			0=& \, x+1 & \, 1-x=& \, x+1 & \, 2-x = & \, x+1
		\end{align*}
		
		Removing the trivial cases $x=0$ and $x=1$, the only solutions are $x\in \{-1,2,\frac{p+1}{2}\}$.
	\end{proof}
	
	Therefore, when $x \not\in \{ -1,2,\frac{p+1}{2} \}$ we make use of the following result.
	
	\begin{theorem}[Perarnau-Perkins, Corollary 6 in \cite{Perarnau-Perkins}] \label{thm:perkins}
		Let $H_{3,6}$ be the point-line incidence graph of the Fano plane, also called the Heawood graph. For any cubic $n$-vertex graph $H$ of girth at least 5,
		$$i(H)^{1/n} \le i(H_{3,6})^{1/14} = 458^{1/14},$$
		with equality if and only if $H$ is a union of copies of $H_{3,6}$.
	\end{theorem}
	
	By combining the previous two results, we obtain the following when $\HH_A$ is bipartite and $C_4$-free.
	
	\begin{corollary}
		If $A=\{0,1,x\}$ and $x \not\in \{-1,2,\frac{p+1}{2}\}$, then 
		$$i(\HH_A) \le 458^{p/7} \approx 2.3995^p.$$
	\end{corollary}
	
	The Heawood graph, the point-line incidence graph of the Fano plane, is the link graph $\HH_A$ when $p=7$ and $x=3$, which means we have equality in one case above. For larger values of $p$, the next result tells us we are not giving away too much by using this upper bound, as it implies $i(\HH_A) \ge (2.38898\dots )^p$. 
	
	Theorem \ref{thm:csikvari} below follows from Theorem 8.1 of \cite{Csikvari}, which follows from a result in \cite{Ruozzi}. Theorem 8.1 of \cite{Csikvari} is more general as it states that the infinite regular tree minimizes the normalized independence polynomial over all regular graphs. Here, we only state its immediate consequence for the total number of independent sets. 
	
	\begin{theorem} \label{thm:csikvari}
		Let $\alpha$ be the unique solution of
		$$ \frac{\alpha}{1-\alpha} = \left( \frac{1-2\alpha}{1-\alpha} \right)^d $$
		in the interval $[0,1/2]$. Let $H$ be a $d$-regular bipartite graph on $n$ vertices. Then
		$$i(H) \ge \left( \frac{(1-\alpha)^{d-1}}{\alpha} \right)^{n/2}.$$
	\end{theorem}
	
	Notice we make use of this result for the $3$-regular graph $\HH_A$. When $d=3$, the unique real solution of $x (1-x)^2 = (1-2x)^3$ is $\alpha \approx 0.24109 \dots$, which implies $i(\HH_A) \ge (2.38898 \dots)^{n}$.

	Now, we restrict ourselves to the case $A=\{0,1,x\}$ with $x\in \{-1,2,\frac{p+1}{2}\}$. We will use the same notation as in Claim \ref{claim:c4}, i.e.,  we write $(Y,t)$ for the vertex corresponding to $t \in \Zp$ in $B$ and, similarly, $(Z,t)$ for $t \in \Zp$ in $C$. In the next claims, we will see that the link graphs $\HH_A$ are isomorphic for $x\in \{-1,2,\frac{p+1}{2}\}$ and then we may assume, without loss of generality, that we have $A=\{-1,0,1\}$ and count the independent sets $i(\HH_A)$ of the link graph $\HH_A$.
	
	\begin{claim}
		The three graphs $\HH_A$ are isomorphic when $A=\{0,1,x\}$ and $x\in \{-1,2,\frac{p+1}{2}\}$.
	\end{claim}
	
	\begin{proof}
		Consider the function $\varphi$ given by  
		$$ \varphi((X,j)) = 
		\begin{cases}
			\left( Y,\frac{j(p+1)}{2} \right) & \text{ if } X=Y \\
			\left( Z,\frac{(j+1)(p+1)}{2} \right) & \text{ if } X=Z
		\end{cases}
		$$
		with the second coordinate taken modulo $p$.  
		
		The function $\varphi$ is a graph isomorphism between the graph obtained for $x=-1$ and the one from $x=(p+1)/2$, since it sends $(Y,j)$ to $(Y,j \cdot \frac{p+1}{2})$ and the neighbors $(Z,j-1), (Z,j), (Z,j+1)$ of $(Y,j)$ for $x=-1$ to $(Z,j \cdot \frac{p+1}{2}), (Z,j \cdot \frac{p+1}{2} + \frac{p+1}{2}), (Z,j\cdot \frac{p+1}{2} + 1)$, neighbors of $(Y,j\cdot \frac{p+1}{2})$ for $x=\frac{p+1}{2}$.
		
		Consider the function $\theta$ given by
		$$ \theta((X,j)) = 
		\begin{cases}
			( Y,j ) & \text{ if } X=Y \\
			( Z,j+1 ) & \text{ if } X=Z
		\end{cases}
		$$
		with the second coordinate taken modulo $p$. Similarly, $\theta$ is a graph isomorphism between the graph obtained for $x=-1$ and the one from $x=2$. 
	\end{proof}
	
	Similarly to Claim \ref{claim:indset_cycle}, we can use recursion to count the independent sets in the graph $\HH_A$ when $A=\{-1,0,1\}$. We obtain the following Claim, which
	we prove, using generating function, in the Appendix.

	\begin{claim} \label{claim:indset_x=-1}
		Let $A=\{-1,0,1\}$. For $p \ge 3$, the link graph $\HH_A$ has 
		$$i(\HH_A) = (1+\sqrt{2})^p+(1-\sqrt{2})^p + 1.$$
	\end{claim}

	For the polynomial terms $Q_{3,1}(p)$ and $Q_{3,2}(p)$ in Theorem \ref{thm:main}, note that if we choose $2$ elements in $A$ we have $3$ possible options for the third element to form a graph isomorphic to the one with $A=\{-1,0,1\}$. We count each triplet $3$ times this way, so we get $Q_{3,1}(p)=3\binom{p}{2}=3 \cdot \frac{3}{p-2} \cdot \binom{p}{3}$ possible sets corresponding to link graphs isomorphic to $\HH_A$ with $A=\{-1,0,1\}$. The number of remaining 3-sets is $Q_{3,2}(p)=3 \binom{p}{3} - Q_{3,1}(p) = 3 \cdot \frac{p-5}{p-2} \cdot \binom{p}{3}$.
	
	In total, the number of triplets with $\min \{ |A|, |B|, |C| \} = 3$ is bounded from above by
	\begin{equation}\label{eq:caseZp,k=3,upper}
		Q_{3,1}(p) \cdot (\sqrt{2}+1)^{p} + Q_{3,2}(p) \cdot 458^{p/7}  + O(p^{6} 2^{p}),
	\end{equation}
	and from below by
	\begin{equation}\label{eq:caseZp,k=3,lower}
		Q_{3,1}(p) \cdot (\sqrt{2}+1)^{p} + Q_{3,2}(p) \cdot (2.38898\dots)^{p} + O(p^{6} 2^{p}).
	\end{equation}

	\section{Small case: $4\le k \le \varepsilon p$} \label{sec:smallcase}
	
	To estimate the triplets with $\min \{ |A|, |B|, |C| \} = k$ for $k \ge 4$ we use the following result. 
	
	\begin{theorem}[Kahn, \cite{Kahn}] \label{thm:kahn}
		If $H$ is a bipartite $d$-regular graph on $n$ vertices, then
		$$i(H) \le i(K_{d,d})^{n/2d} = (2^{d+1}-1)^{n/2d}.$$
	\end{theorem}
	
	Hence, the number of independent sets in a 4-regular bipartite graph on $2p$ vertices is at most $31^{p/4}$, and in a $k$-regular bipartite graph on $2p$ vertices for $k\ge 5$ is at most $63^{p/5}$.
	
	Therefore, the number of triplets for which $4\le k \le \varepsilon p$ is at most 
	$$ 3\binom{p}{4} 31^{p/4} + \sum_{5\le k \le \varepsilon p}  3 \binom{p}{k} \cdot 63^{p/5} \le 3\binom{p}{4} 31^{p/4} + 3 \cdot 63^{p/5} \cdot \sum_{k \le \varepsilon p}  \binom{p}{k}.$$
	
	\noindent For the sum $\sum\limits_{k \le \varepsilon p}  \binom{p}{k}$, we use the standard binomial estimates $\binom{n}{k} \le (en/k)^k$ to obtain 
	$$\sum\limits_{k \le \varepsilon p}  \binom{p}{k} \le 2 \cdot \left( e^{\varepsilon (1 + \log(1/\varepsilon))} \right)^p .$$
	
	\noindent We conclude that, for $\varepsilon < 1/10^3$, the number of triplets with $4 \le \min \{ |A|, |B|, |C| \} \le \varepsilon p$ is at most  
	\begin{equation}\label{eq:caseZp,ksmall}
		\left( 3\binom{p}{4} +o(1) \right) 31^{p/4} .
	\end{equation}
	
	\section{Big case: $k \ge \varepsilon p$} \label{sec:bigcase}
	
	To count the triplets $(A,B,C)$ with $|A|,|B|,|C| \ge \varepsilon p$, we use the method of hypergraph containers developed by Balogh, Morris, and Samotij \cite{balogh2015independent}, and, independently, by Saxton and Thomasson \cite{saxton2015hypergraph}. Since its development, the method has found several applications, in particular for counting sum-free sets \cite{balogh2015sharp}. We highlight that a container theorem for sum-free subsets of abelian groups was found earlier by Green and Ruzsa (see Proposition 2.1 of \cite{Green-Ruzsa}) using Fourier analysis. Here, we need a 3-partite variant of their result, hence we make use of the container lemma for general hypergraphs.
	
	Throughout this section, we will state and prove the results for a general abelian group $G$ and at the end of the section, we discuss the differences when $G=\Zp$.
	
	\subsection{Supersaturation}
	
	To use the container method, we need the supersaturation result (Proposition \ref{prop:ssat2}) below. We highlight that Propositions \ref{prop:ssat2} and \ref{prop:ssat1} are very similar in essence to Remark 20 of \cite{Semchankau}. 
	
	\begin{definition}
		Given $A,B\subseteq G$ with $G$ an abelian group and given $\varepsilon\in (0,1)$, we write $A+_{\varepsilon}B$ for the set of $x\in G$ having at least $\varepsilon |G|$ representations as a sum $a+b=x$ with $a\in A$ and $b\in B$.  
	\end{definition}
	
	As in \cite{Semchankau}, we obtain a supersaturation result as a corollary of the following Lemma.
	
	\begin{lemma}[Theorem 19 in \cite{Semchankau}]\label{lem:ssat2}
		Let $G$ be an abelian group, and $A, B \subset G$ be sets. Let $\varepsilon \in (0,1)$ be such that $\sqrt{\varepsilon} |G| < |A|, |B|$. Let $H$ be a maximal proper subgroup of $G$. Then 
		$$ |A+_{\varepsilon} B| \ge \min\{ |G|, |A|+|B|-|H| \} -3 \sqrt{\varepsilon} |G| .$$
	\end{lemma}
	
	\begin{proposition}\label{prop:ssat2}
		Fix $\varepsilon \in (0,1)$. Then for sufficiently large $N$ the following holds. Let $G$ be an abelian group of order $N$ and let $A, B, C \subset G$ be sets. There exists $\sigma >0$ such that if $|A|,|B|,|C| \ge \varepsilon N$ and $|A|+|B|+|C| \ge \left( \frac{3}{2} + \varepsilon \right) N$ then there are at least $\sigma N^2$ sums $x+y=z$ with $x\in A$, $y\in B$, and $z \in C$.
	\end{proposition}
	
	\begin{proof}
		We know that $|A|,|B|\geq \varepsilon N > (\varepsilon/4) N$. By Lemma \ref{lem:ssat2}, we get 
		$$|A+_{\varepsilon^2/16}B|\geq |A|+|B|-|H|-\left( \frac{3\varepsilon}{4}\right) N .$$
		Since $H$ is a proper subgroup, we have $|H|\leq N/2$ and then 
		$$ |A+_{\varepsilon^2/16}B|\geq |A|+|B|-\frac{N}{2} -\left( \frac{3\varepsilon}{4}\right) N .$$ 
		We obtain
		\begin{align*}
			|C\cap (A+_{\varepsilon^2/16}B)| \geq &\  |C|+|(A+_{\varepsilon^2/16}B)|-N \ge \\
			\geq & \left( \left( \frac{3}{2}+\varepsilon \right)N-|A|-|B| \right)+\left( |A|+|B|-\frac{N}{2} -\left( \frac{3\varepsilon}{4}\right) N \right)-N = \left( \frac{\varepsilon}{4} \right) N.
		\end{align*}
		Then there are at least $(\varepsilon^3/64)N^2$ sums of the form $x+y=z$ with $x\in A$, $y\in B$ and $z\in C$. This concludes the proof of the proposition with $\sigma=\varepsilon^3/64$.
	\end{proof}

	\subsection{Containers}
	
	Now, we state the Hypergraph Container Lemma (Theorem \ref{thm:container}) and obtain the upper bound on the number of triplets $(A,B,C)$ with $|A|,|B|,|C| \ge \varepsilon p$. First, we need some definitions.
	
	\begin{definition}
		Fix a $r$-regular hypergraph $\G$ with average degree $d$; fix $v\in V(\G)$; fix $2\leq j\leq |V(\G)|$ and $\tau>0$. Define 
		
		\[
		d^{(j)}(v):=\max\{|S|: v\in A\subseteq V(\G), |S|=j\}.
		\]
		
		\noindent Define $\delta_j$ with the equation
		\[
		\delta_{j}\tau^{j-1}d|V(\G)|=\sum_{v}d^{(j)}(v).
		\]
		
		\noindent Finally, define the \textit{co-degree} function $\delta(\G,\tau)$ with 
		\[
		\delta(\G,\tau):=2^{\binom{r}{2}-1}\sum_{j=2}^{r}2^{-\binom{j-1}{2}}\delta_j.
		\]
		If $d=0$ define $\delta(\G,\tau)=0$.
	\end{definition}

	\begin{theorem}[Corollary 3.6 in \cite{saxton2015hypergraph}]\label{thm:container}
		Let $\G$ be an $r$-uniform hypergraph with vertex set $[N]$. Let $0< \varepsilon, \tau < 1/2$. Suppose that $\tau < 1/(200 \cdot r \cdot r!^2)$ and $\delta(\G, \tau) \le \varepsilon / (12r!)$. Then there exists $c=c(r) \le 1000 \cdot r \cdot r!^3$ and a collection $\C$ of vertex subsets such that 
		\begin{enumerate}
			\item[(i)] every independent set in $\G$ is a subset of some $D \in \C$;
			\item[(ii)] for every $D \in \C$, $e(\G[D]) \le \varepsilon e(\G)$;
			\item[(iii)] $\log |\C| \le cN\tau \cdot \log (1/\varepsilon) \cdot \log (1/ \tau)$.
		\end{enumerate}
	\end{theorem}
	
	\begin{claim} \label{claim:generalbigcase}
		Fix $\varepsilon \in (0,1)$. Then for sufficiently large $N$, the number of triplets $A,B,C \subset G$ with $\min \{ |A|, |B|, |C| \} \ge \varepsilon N$ is at most $2^{\left( \frac{3}{2} + 2\varepsilon \right) N}$.
	\end{claim}
	
	\begin{proof}
		Fix $\varepsilon \in (0,1)$. Using Proposition \ref{prop:ssat2}, we obtain a $\sigma>0$. We choose $\varepsilon_0>0$ sufficiently small, in particular with $\varepsilon_0 < \sigma$. 
		We apply Theorem \ref{thm:container} to the 3-uniform hypergraph $\HH$. First, realize that every vertex in $\HH$ has degree $N$, hence the average degree of $\HH$ is $N$. For any $v\in V(\HH)$, we have $d^{(2)}(v)=d^{(3)}(v)=1$. Therefore, 
		$$ \delta_2 \tau^1 \cdot 3 N^2 = \sum_{v}d^{(2)}(v) = 3N \quad \Rightarrow \quad \delta_2=1/N\tau, $$ 
		$$ \delta_3 \tau^2 \cdot 3N^2=\sum_{v}d^{(3)}(v)=3N \quad \Rightarrow \quad \delta_3=1/N\tau^2. $$
		
		\noindent We obtain that
		
		\[
		\delta(\HH,\tau) = \frac{4}{N\tau}+ \frac{2}{N\tau^2}.
		\]
		
		Take $\tau = 400 \varepsilon_0^{-1} N^{-1/2}$, so that $\delta(\HH,\tau) = \frac{4}{N\tau}+ \frac{2}{N\tau^2} \le \frac{\varepsilon_0}{72}$.
		We can use Theorem \ref{thm:container} with the given $\tau$ and $\varepsilon_0$ to get a collection of containers $\C$ with size 
		\[
		\log |\C|\leq c \cdot \varepsilon_0^{-1} N^{1/2}\cdot \log(1/\varepsilon_0)\cdot \log(1/\tau) \le C_{\varepsilon_0} N^{1/2}\log(N) ,
		\]
		with $C_{\varepsilon_0}$ some positive constant depending on $\varepsilon_0$.
		
		Now, suppose we fix a container $D\in\C$ with $D\cap X=A^*$, $D\cap Y=B^*$ and $D\cap Z=C^*$. 
		Note that the independent sets contained in containers with $\min\{|A^*|,|B^*|,|C^*|\} < \varepsilon |G|$ are not included in the count for this claim, we already counted them in the previous sections.
		Assume, then, that $|A^*|,|B^*|,|C^*|\geq \varepsilon N$.
		
		If $|A^*|+|B^*|+|C^*|\geq (3/2+\varepsilon)N$, then 
		by Proposition \ref{prop:ssat2}, we get a contradiction with $e(\HH[D]) \geq \sigma N^2$ and $e(\HH[D])\leq \varepsilon_0 e(\HH)<\sigma N^2$. Then $|A^*|+|B^*|+|C^*|\leq \left( \frac{3}{2}+\varepsilon \right)N$, and the number of independent sets contained in $D$ is at most the number of subsets in it, $2^{(3/2+\varepsilon)N}$. Therefore, for sufficiently small $\varepsilon_0$, we have at most
		\[
		2^{(3/2+\varepsilon)N+C_{\varepsilon_0} N^{1/2}\log(N)}<2^{(3/2+2\varepsilon)N}
		\]
		independent sets.
	\end{proof}
	
	\begin{remark}\label{rem:index2} \normalfont
		As mentioned in Remark 22 of \cite{Semchankau}, the exponent 3/2 in Claim \ref{claim:generalbigcase} is best possible. If $H$ is a subgroup of index 2, $A,B \subset H$ and $C \subset G \backslash H$, then $(A,B,C)$ is a sum-free triple. Therefore, in the case $G$ has a subgroup of index 2, there are at least $2^{3N/2}$ triplets. 
	\end{remark}
	
	\section{Proof of Theorem \ref{thm:main}} \label{sec:proof}
	
	When $G=\Zp$, there is only one proper subgroup $H$ for which $|H|=1$. Hence, adapting the proof of Proposition~\ref{prop:ssat2} we obtain the following. 
	
	\begin{proposition}\label{prop:ssat1}
		Fix $\varepsilon \in (0,1)$. Then for sufficiently large $p$ the following holds. Let $A, B, C \subset \Zp$ be sets. There exists $\sigma >0$ such that if $ \varepsilon p < |A|, |B|, |C|$ and $|A|+|B|+|C| \ge \left( 1 + \varepsilon \right) p$ then there are at least $\sigma p^2$ sums $x+y=z$ with $x\in A$, $y\in B$, and $z \in C$.
	\end{proposition}
	
	An alternative way to prove Proposition \ref{prop:ssat1} directly is by using the following Lemma.
	
	\begin{lemma}[Theorem 8 in \cite{Semchankau}]
		Let $A, B \subset \Zp$ be sets and $\varepsilon \in (0,1)$ be a real number, such that $\sqrt{\varepsilon} p < |A|, |B|$. Then
		$$ |A+_\varepsilon B | \ge \min\{ p, |A|+|B| \} - 2p \sqrt{\varepsilon}.$$
	\end{lemma}
	
	\noindent Similarly to the proof of Claim \ref{claim:generalbigcase}, but now using Proposition \ref{prop:ssat1}, we obtain the following.
	
	\begin{claim}
		Fix $\varepsilon \in (0,1)$. Then for sufficiently large $p$, the number of sum-free triplets $A,B,C \subset \Zp$ with $\min \{ |A|, |B|, |C| \} \ge \varepsilon p$ is at most 
		\begin{equation}\label{eq:caseZp,kbig}
			2^{(1+2\varepsilon) p}.
		\end{equation}
	\end{claim}
	
	Now that we covered all the possible cases for $k$, we can proceed to prove Theorem \ref{thm:main}.
	
	\begin{proof}[Proof of Theorem \ref{thm:main}] 
		Using the same setting described in the introduction, we count the ordered sum-free triplets $(A,B,C)$ such that $\min\{|A|,|B|,|C|\}=k$. We have an upper bound for the number of triplets in  the cases $k=0,1,2$ on \eqref{eq:caseZp,k=0,1,2}; for $k=3$ on \eqref{eq:caseZp,k=3,upper}; for $4\leq k\leq \varepsilon N$ on $\eqref{eq:caseZp,ksmall}$ and finally for $k\ge \varepsilon N$ on \eqref{eq:caseZp,kbig}. Summing up the upper bounds, for $\varepsilon< 1/10^3$, we obtain that there are at most 
		$$ 3\cdot 4^p + 3p \cdot 3^p + 3\binom{p}{2} \left( \frac{1+\sqrt{5}}{2} \right)^{2p} + Q_{3,1}(p) \cdot (1+\sqrt{2})^p + Q_{3,2}(p) \cdot 458^{p/7} + 3\binom{p}{4} (1+o(1)) \cdot 31^{p/4} $$
		many triplets, where we add the missing terms $2^{(1+2\varepsilon)p}$ and $O(p^62^p)$ to the $o(31^{p/4})$ term.
		
		Now, for the lower bound, the triplets for the cases $k=0,1,2$ are counted on \eqref{eq:caseZp,k=0,1,2}. For the triplets in the case $k=3$, we obtained a lower bound on \eqref{eq:caseZp,k=3,lower}. The amount of triplets that are overcounted in both calculations is upper bounded by $9\binom{p}{3}^2 2^{p}$. This error term is omitted in the final result, since we truncate the value of $(1-\alpha)^2/\alpha$ and write $2.38898$ instead. The sum of the two lower bounds shows there are at least 
		$$ 3\cdot 4^p + 3p \cdot 3^p + 3\binom{p}{2} \left( \frac{1+\sqrt{5}}{2} \right)^{2p} + Q_{3,1}(p) \cdot (1+\sqrt{2})^p + Q_{3,2}(p) \cdot 2.38898^{p}$$
		many triplets.
	\end{proof}

	\section{General abelian groups} \label{sec:generalgroup}
	
	Let $(G,+)$ be an additive abelian group of order $N$. Using the same methods as in the case $G=\Zp$, we now prove the following similar result.
	
	\textbf{Theorem \ref{thm:generalcase}.}
	\textit{
		Fix $\varepsilon \in (0,1)$. Then for sufficiently large $N$ the following holds. Let $(G,+)$ be an additive abelian group of order $N$. The number of ordered triplets $A,B,C\subseteq G$ such that there is no triplet $a\in A$, $b\in B$ and $c\in C$ with $a+b=c$ is at most}
	$$ 3\cdot 4^N + 3N \cdot 3^N + 2^{(3/2+3\varepsilon)N} . $$
	\begin{proof}
		Assume that for a fixed integer $k$ we have $\min\{|A|,|B|,|C|\}=k$ and $|A|=k$. When $k=0$ and $k=1$ we obtain the same number of triplets $A,B,C$ as in $G=\Zp$. We get exactly 
		
		\begin{equation}\label{eq:generalcasek=0}
			3 \cdot 4^N - 3 \cdot 2^N + 1 
		\end{equation}
		triplets for $k=0$ and
		\begin{equation}\label{eq:generalcasek=1}
			3N\cdot (3^N-1) - 3N^2 \cdot (2^{N-1}-1) + N^2(N-1) 
		\end{equation}
		triplets for $k=1$.
		
		As in Section \ref{sec:smallcase}, we use Theorem \ref{thm:kahn} to upper bound the triplets with $2 \le k \le \varepsilon N$. The number of independent sets in a 2-regular bipartite graph on $2N$ vertices is at most $7^{N/2}$, and in a $k$-regular bipartite graph on $2N$ vertices for $k\ge 3$ is at most $15^{N/3}$.
		Therefore, for sufficiently small $\varepsilon$, the number of triplets with $2 \le \min \{ |A|, |B|, |C| \} \le \varepsilon N$ is at most  
		\begin{equation}\label{eq:generalcasesmall}
			\left( 3\binom{N}{2} + o(1) \right) 7^{N/2} .
		\end{equation} 
		
		\noindent For $k \ge \varepsilon N$, by Claim \ref{claim:generalbigcase}, we obtain that the number of triplets is at most 
		\begin{equation}\label{eq:generalcasebig}
			2^{(3/2+2 \varepsilon )N}.
		\end{equation}
		The upper bounds in \eqref{eq:generalcasek=0}, \eqref{eq:generalcasek=1}, \eqref{eq:generalcasesmall} and \eqref{eq:generalcasebig} give together the desired upper bound for the total number of triplets. Note that even that we obtained the bound \eqref{eq:generalcasesmall} only for small $\varepsilon$, the final result for all $\varepsilon \in (0,1)$ follows from it.
	\end{proof}

	\subsection{More precise computations in the general case}
	
	Write $N=p_1^{\alpha_1} \cdot \ldots \cdot p_t^{\alpha_t}$ where $p_1 < \ldots < p_t$ are prime numbers and $\alpha_i \ge 1$. Then, we can write $G = G_1 \times \cdots \times G_t$, where $|G_i|=p_i^{\alpha_i}$. Given $p_1$, we can get better estimates for the upper and lower bounds in the number of ordered triplets $(A,B,C)$. We highlight that the following Theorem implies Theorem \ref{thm:generalcase}. 
	
	\begin{theorem} \label{thm:p1}
		Fix $\varepsilon \in (0,1)$. Then for sufficiently large $N$ the following holds. Let $(G,+)$ be an additive abelian group of order $N=p_1^{\alpha_1} \cdot \ldots \cdot p_t^{\alpha_t}$. Let $\varepsilon \in (0,1)$. The number of ordered triplets $A,B,C\subseteq G$ such that there is no triplet $a\in A$, $b\in B$ and $c\in C$ with $a+b=c$ is at most
		$$ 3\cdot 4^N + 3N \cdot 3^N + 2^{(1+1/p_1+2 \varepsilon )N} + 3\binom{N}{2} (\phi^{2p_1}+\phi^{-2p_1})^{N/p_1} + 3\binom{N}{3}(1+o(1)) 15^{N/3} , $$
		and at least 
		$$ 3\cdot 4^N + 3N \cdot 3^N + 2^{(1+1/p_1)N} + \frac{3N}{2} (\phi^{2p_1}+\phi^{-2p_1})^{N/p_1} . $$
	\end{theorem}
	
	\begin{proof}
		We proceed as in the proof of Theorem \ref{thm:generalcase} and obtain the same bounds when $k =0,1$.
		When $k=2$, we write $A=\{x,y\}$. In this case, we have that $\HH_A$ is the disjoint union of cycles of length $2 \cdot o( x-y )$, where $o(x-y)$ denotes the order of the element $x-y$ in $G$. We count the independent sets in $\HH_A$ using Claim \ref{claim:indset_cycle}.  
		
		We conclude that there are 
		\begin{equation} \label{eq:p1k=2}
			\sum_{t=2}^{N} \frac{a_t}{N-1} \cdot 3 \binom{N}{2} (\phi^{2t}+\phi^{-2t})^{N/t} = \frac{3N}{2} \phi^{2N} \sum_{t=2}^{N} a_t \cdot (1+ \phi^{-4t})^{N/t}
		\end{equation}
		triplets with $\min \{ |A|, |B|, |C| \} = 2$, where $a_t$ is the number of elements of order $t$ in $G$. 
		As $(1+\phi^{-4x})^{N/x}$ is a decreasing function in $x$ and $a_t=0$ for $1<t<p_1$,
		the maximum value of the sum in \eqref{eq:p1k=2} is obtained when $a_{p_1}=N-1$. Thus, for $k=2$, the number of triplets is at most  
		\begin{equation} \label{eq:p1_up_k=2}
			3\binom{N}{2} (\phi^{2p_1}+\phi^{-2p_1})^{N/p_1} .
		\end{equation}
		Moreover, as $p_1$ divides $|G|$, there is at least one element of order $p_1$ in $G$. Hence, we have at least  
		\begin{equation} \label{eq:p1_lw_k=2}
			\frac{3N}{2} (\phi^{2p_1}+\phi^{-2p_1})^{N/p_1}
		\end{equation}
		many triplets.
		
		When $3\le k \le \varepsilon N$ and $\varepsilon$ is small, we follow the same method as in Section \ref{sec:smallcase}, which gives that there are at most
		\begin{equation} \label{eq:p1_k=3}
			\left( 3\binom{N}{3} +o(1) \right) 15^{N/3}    
		\end{equation}
		many triplets.
		
		For the case $k \ge \varepsilon N$, we note that the maximum order of a proper subgroup is $N / p_1$. Similarly to the proof of Claim \ref{claim:generalbigcase}, by Lemma \ref{lem:ssat2} with $H$ of order $|H|=N/p_1$, we obtain that there are at most 
		\begin{equation} \label{eq:p1_up_bigcase}
			2^{(1+1/p_1+2\varepsilon)N}
		\end{equation}
		many triplets.
		
		For a lower bound, notice that if $H$ is a subgroup of index $p_1$, $A,B \subset H$ and $C \subset G \backslash H$, then $(A,B,C)$ is a sum-free triple. Therefore there are at least 
		\begin{equation} \label{eq:p1_lw_bigcase}
			2^{(1+1/p_1)N}
		\end{equation}
		many triplets. 
		
		Similarly to the proof of Theorem \ref{thm:generalcase}, we conclude by adding the bounds for all possible $k$. From \eqref{eq:generalcasek=0}, \eqref{eq:generalcasek=1}, \eqref{eq:p1_up_k=2}, \eqref{eq:p1_k=3} and \eqref{eq:p1_up_bigcase}, we obtain that for small enough $\varepsilon$ there are at most 
		$$ 3\cdot 4^N + 3N \cdot 3^N + 2^{(1+1/p_1+2 \varepsilon )N} + 3\binom{N}{2} (\phi^{2p_1}+\phi^{-2p_1})^{N/p_1} + 3\binom{N}{3}(1+o(1)) 15^{N/3} $$
		triplets, and from \eqref{eq:generalcasek=0}, \eqref{eq:generalcasek=1}, \eqref{eq:p1_lw_k=2} and \eqref{eq:p1_lw_bigcase}, we obtain that there are at least 
		$$ 3\cdot 4^N + 3N \cdot 3^N + 2^{(1+1/p_1)N} + \frac{3N}{2} (\phi^{2p_1}+\phi^{-2p_1})^{N/p_1} $$ 
		many triplets. Finally, note that the upper bound for all $\varepsilon \in (0,1)$ follows from the upper bound for small $\varepsilon$.
	\end{proof}

	At last, we discuss the estimates we obtain from Theorem \ref{thm:p1} for small values of $p_1$.
	For $p_1=2$, we have more triplets coming from the $k\ge \varepsilon N$ case, since 
	$$2^{1+1/2} = 8^{1/2} > 7^{1/2} = (\phi^{4}+\phi^{-4})^{1/2} \approx 2.64575 > 15^{1/3} \approx 2.466.$$
	
	\noindent For $p_1 = 3$, we have more triplets coming from $k=2$ case, then $k\ge \varepsilon N$, then $k=3$, since
	$$(\phi^{6}+\phi^{-6})^{1/3} = 18^{1/3} > 16^{1/3} = 2^{1+1/3} > 15^{1/3}.$$
	
	\noindent For $p_1 = 5$, we have more triplets coming from $k=2$ case, then $k=3$, then $k\ge \varepsilon N$, since
	$$(\phi^{10}+\phi^{-10})^{1/5} = 123^{1/5} \approx 2.618 > 
	15^{1/3} \approx 2.466 > 
	2^{1+1/5} \approx 2.297 .$$
	
	\section*{Acknowledgements}
	
	We thank the anonymous referee for their useful comments that improved the presentation of this manuscript.
	
	\bibliographystyle{amsplain} 
	\bibliography{refs} 
	
	\appendix
	\section{Proof of Claim \ref{claim:indset_x=-1}}
	
	\textbf{Claim \ref{claim:indset_x=-1}.} 
	\textit{ 
		Let $A=\{-1,0,1\}$. For $p \ge 3$, the link graph $\HH_A$ has} 
	$$i(\HH_A) = (1+\sqrt{2})^p+(1-\sqrt{2})^p + 1.$$
	\begin{proof}
		We count the independent sets in $\HH_A$ recursively.
		First, assume $(Y,0)$ is part of the independent set. So $(Z,-1),(Z,0)$ and $(Z,1)$ are not in the independent set $I$. Let $G_p$ denote the graph obtained by removing $(Y,0)$ and $N((Y,0))$ from the original graph $\HH_A$. Let $a_p$ denote the number of independent sets on $G_p$.
		
		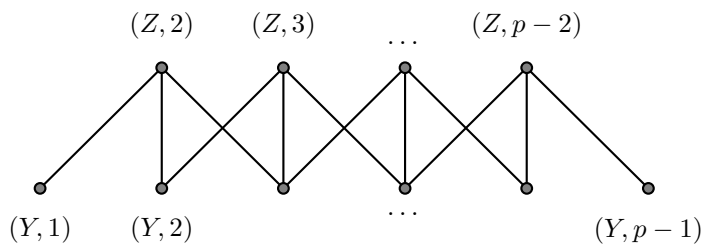
\begin{figure}[H]
			\centering
			\begin{tikzpicture}[thick,scale=0.8]
				\draw[color=white, use as bounding box] (0,-1) rectangle (10,3.3);
				\draw (8,2.7) node [fill=black!0,draw=black!0] {$(Z,p-2)$};
				\draw (10,-0.7) node [fill=black!0,draw=black!0] {$(Y,p-1)$};
				\foreach \x/\y in {0/1,1/2,2/3,3/4} 
				{\draw (2*\x,0) -- (2*\y,2);
					\draw (2*\y,2) -- (2*\x+2,0);
					\draw (2*\y,2) -- (2*\x+4,0);}
				\draw (0,0) node [label=below:{$(Y,1)$}] {};
				\draw (2,0) node [label=below:{$(Y,2)$}] {};
				\draw (4,0) node [] {};
				\draw (6,0) node [label=below:{$\dots$}] {};
				\draw (8,0) node [] {};
				\draw (10,0) node [] {};
				\draw (2,2) node [label=above:{$(Z,2)$}] {};
				\draw (4,2) node [label=above:{$(Z,3)$}] {};
				\draw (6,2) node [label=above:{$\dots$}] {};
				\draw (8,2) node [] {};
			\end{tikzpicture}
			\caption{The graph $G_p$.}
			\label{fig:Gp}
		\end{figure}
		
		From here we will calculate $a_p$ using recurrence. Assume $(Y,1)$ is in the independent set. The graph obtained by erasing $(Y,1)$ and $N((Y,1))$ from $G_p$ is isomorphic to $G_{p-1}$. If neither $(Y,1)$ and $(Z,2)$ are part of the independent set, erasing them we get again a graph isomorphic to $G_{p-1}$.
		
		The remaining case is when $(Z,2)$ belongs to the independent set. Let $H_p$ be the graph obtained after erasing $(Z,2)$ and its neighbors from $G_p$, and let $b_p$ be the number of independent sets in $H_p$. 
		
		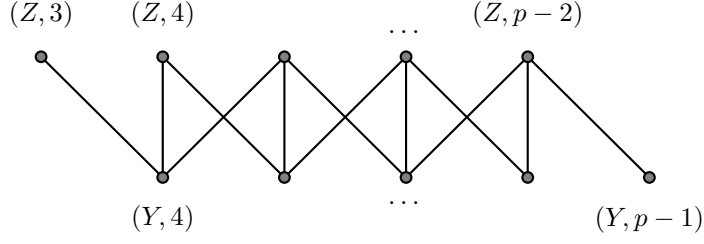
\begin{figure}[H]
			\centering
			\begin{tikzpicture}[thick,scale=0.8]
				\draw[color=white, use as bounding box] (0,-1) rectangle (10,3.3);
				\draw (8,2.7) node [fill=black!0,draw=black!0] {$(Z,p-2)$};
				\draw (10,-0.7) node [fill=black!0,draw=black!0] {$(Y,p-1)$};
				\foreach \x/\y in {1/2,2/3,3/4} 
				{\draw (2*\x,0) -- (2*\y,2);
					\draw (2*\y,2) -- (2*\x+2,0);
					\draw (2*\y,2) -- (2*\x+4,0);}
				\draw (4,0) -- (2,2);
				\draw (2,0) -- (2,2);
				\draw (2,0) -- (0,2);
				\draw (0,2) node [label=above:{$(Z,3)$}] {};
				\draw (2,0) node [label=below:{$(Y,4)$}] {};
				\draw (4,0) node [] {};
				\draw (6,0) node [label=below:{$\dots$}] {};
				\draw (8,0) node [] {};
				\draw (10,0) node [] {};
				\draw (2,2) node [label=above:{$(Z,4)$}] {};
				\draw (4,2) node [] {};
				\draw (6,2) node [label=above:{$\dots$}] {};
				\draw (8,2) node [] {};
			\end{tikzpicture}
			\caption{The graph $H_p$.}
			\label{fig:Hp}
		\end{figure}
		
		Then we have
		\[
		a_p=2a_{p-1}+b_p. 
		\]
		
		Now we proceed with the same recursive method for $b_p$. If $(Z,3)$ is part of the independent set, by erasing it and its neighbors we get a graph isomorphic to $H_{p-1}$. If neither $(Z,3)$ nor $(Y,4)$ belong to $I$, by erasing them we get a graph isomorphic to $H_{p-1}$. 
		
		Finally, if $(Y,4)$ belongs to the independent set we get a graph isomorphic to $G_{p-4}$, so 
		\[
		b_{p}=2b_{p-1}+a_{p-4}.
		\]
		We can calculate $a_3=4, a_4=9$ and $b_4=1, b_5=3, b_6=8$. Define $a_2:=2$, so the equation $b_6=2b_5+a_2$ holds. Let $F(x)=\sum_{p\geq 2} a_p x^p$ and $G(x)=\sum_{p\geq 4} b_p x^p$.
		
		The equation $b_{p}=2b_{p-1}+a_{p-4}$ holds for $p\geq 6$. The equation $a_p=2a_{p-1}+b_p$ holds for $p\geq 4$. 
		
		From the equation $b_{p}=2b_{p-1}+a_{p-4}$ and the first terms, we have $F(x)=2x F(x)+G(x)+2x^2$.
		So $$G(x)=F(x)-2xF(x)-2x^2.$$
		From the equation $b_{p}=2b_{p-1}+a_{p-4}$ and the first terms, we have 
		\[
		G(x)=2x G(x)+x^4 F(x)+x^4+x^5.
		\]
		\noindent Substituting, we get
		\[
		F(x)-2xF(x)-2x^2=2xF(x)-4x^2 F(x)-4x^3+x^4 F(x)+x^4+x^5
		\]
		\[
		F(x)\left(1-4x+4x^2-x^4\right)=x^4+x^5-4x^3+2x^2
		\]
		\[
		F(x)=-\frac{x^4+x^5-4x^3+2x^2}{(x-1)^2(x+1+\sqrt{2})(x+1-\sqrt{2})}.
		\]
		\[
		F(x)=-x^2\frac{x^2+2x-2}{(x-1)(x+1+\sqrt{2})(x+1-\sqrt{2})}.
		\]
		\noindent From here we can use partial fractions, the fraction is irreducible. We get $$ a_p=\frac{1}{4}(\sqrt{2}+1)^p+\frac{1}{4}(1-\sqrt{2})^p+ \frac{1}{2} \text{ for } p\geq 2 .$$
		
		We return to the initial setting. Assume $(Z,0)$ is in the independent set. After erasing $(Z,0)$ and its neighbors from $\HH_A$, we get a graph isomorphic to $G_p$ and it has $a_p$ independent sets. 
		
		Assume neither $(Y,0)$ or $(Z,0)$ belong to the independent set. Let $G^*_p$ be the graph obtained after deleting $(Y,0)$ and $(Z,0)$ from $\HH_A$. And let $c_p$ be the number of independent sets in $G^*_p$. 
		\begin{figure}[H]
			\centering
			\begin{tikzpicture}[thick,scale=0.8]
				\draw[color=white, use as bounding box] (2,-1) rectangle (8,3);
				\draw (8,2.7) node [fill=black!0,draw=black!0] {$(Z,p-1)$};
				\draw (8,-0.7) node [fill=black!0,draw=black!0] {$(Y,p-1)$};
				\foreach \x/\y in {1/1,2/2,3/3,4/4} 
				{\draw (2*\x,0) -- (2*\y,2);}
				\foreach \x/\y in {1/1,2/2,3/3}
				{\draw (2*\y,2) -- (2*\x+2,0);}
				\foreach \x/\y in {2/2,3/3,4/4}
				{\draw (2*\y,2) -- (2*\x-2,0);}
				\draw (2,0) node [label=below:{$(Y,1)$}] {};
				\draw (4,0) node [label=below:{$(Y,2)$}] {};
				\draw (6,0) node [label=below:{$\dots$}] {};
				\draw (8,0) node [] {};
				\draw (2,2) node [label=above:{$(Z,1)$}] {};
				\draw (4,2) node [label=above:{$(Z,2)$}] {};
				\draw (6,2) node [label=above:{$\dots$}] {};
				\draw (8,2) node [] {};
			\end{tikzpicture}
			\caption{The graph $G_p^*$.}
			\label{fig:Gp*}
		\end{figure}
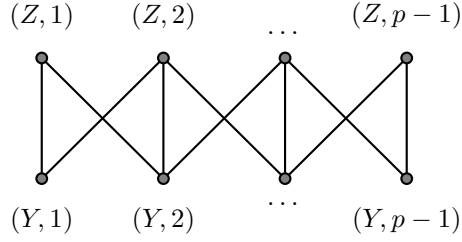
		If $(Y,1)$ is on the independent set, we erase it and its two neighbors so we get a graph and call it $H^{*}_{p}$. If $(Z,1)$ is on the independent set, we erase it we get a graph isomorphic to $H^{*}_p$. Call $d_p$ the number of independent sets in $H^{*}_{p}$.
		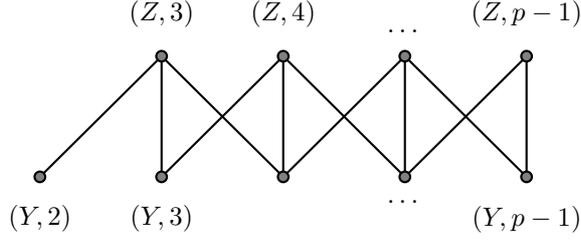
\begin{figure}[H]
			\centering
			\begin{tikzpicture}[thick,scale=0.8]
				\draw[color=white, use as bounding box] (0,-1) rectangle (8,3);
				\draw (8,2.7) node [fill=black!0,draw=black!0] {$(Z,p-1)$};
				\draw (8,-0.7) node [fill=black!0,draw=black!0] {$(Y,p-1)$};
				\foreach \x/\y in {0/1,1/2,2/3} 
				{\draw (2*\x,0) -- (2*\y,2);
					\draw (2*\y,2) -- (2*\x+2,0);
					\draw (2*\y,2) -- (2*\x+4,0);}
				\draw (8,0) -- (8,2);
				\draw (6,0) -- (8,2);
				\draw (0,0) node [label=below:{$(Y,2)$}] {};
				\draw (2,0) node [label=below:{$(Y,3)$}] {};
				\draw (4,0) node [] {};
				\draw (6,0) node [label=below:{$\dots$}] {};
				\draw (8,0) node [] {};
				\draw (2,2) node [label=above:{$(Z,3)$}] {};
				\draw (4,2) node [label=above:{$(Z,4)$}] {};
				\draw (6,2) node [label=above:{$\dots$}] {};
				\draw (8,2) node [] {};
			\end{tikzpicture}
			\caption{The graph $H_p^*$.}
			\label{fig:Hp*}
		\end{figure}
		Finally, if neither of $(Y,1)$ or $(Z,1)$ is on the independent set, we get a graph isomorphic to $G^*_{p-1}$. We obtained the recursion
		\[
		c_p=2d_{p}+c_{p-1}.
		\]
		Now, to count how many independent sets there are in $H^{*}_p$, we have that if $(Y,2)$ is not on $I$ then erasing it gives a graph isomorphic to $G^{*}_{p-2}$. If $(Y,2)$ is on $I$, erasing it and its neighbor we get a graph isomorphic to $H^{*}_{p-1}$. Therefore
		\[
		d_p=c_{p-2}+d_{p-1}.
		\]
		We realize $c_2=3, c_3=7$ and $d_3=2, d_4=5$.
		Let $F^{*}(x)=\sum_{p\geq 2}c_p x^p$ and $G^{*}(x)=\sum_{p\geq 3}d_p x^p$. The recursion $c_p=2d_{p}+c_{p-1}$ works for $p \geq 3$. In generating function form, the first recursion becomes
		\[
		F^{*}(x)=2G^{*}(x)+xF^{*}(x)+3x^2.
		\]
		\[
		G^{*}(x)=(F^{*}(x)-xF^{*}(x)-3x^2)/2.
		\]
		The recursion $d_p=c_{p-2}+d_{p-1}$ holds for $p\geq 4$ and it becomes
		\[
		G^*(x)=x^2F(x)+xG(x)+2x^3.
		\]
		Substituting we obtain
		\[
		F^{*}(x)=-\frac{x^3+3x^2}{-1+2x+x^2}.
		\]
		Again using partial fractions, we get 
		$$c_p=\frac{1}{2}(\sqrt{2}+1)^{p}+\frac{1}{2}(1-\sqrt{2})^p \text{ for } p\geq 3.$$
		We have that the total number of independent sets is 
		$$ 2a_p+c_p=(\sqrt{2}+1)^{p}+(1-\sqrt{2})^p+1 . \eqno \qedhere$$
	\end{proof}
\end{document}